\documentclass[12pt,a4paper]{article}
\usepackage{amssymb}
\usepackage{epsfig}
\usepackage{amsmath} 
\usepackage{esint}
\usepackage{amsthm}
\usepackage{subfigure}
\usepackage{url}
\usepackage{graphicx}
\usepackage[T1]{fontenc}
\usepackage[utf8]{inputenc}
\usepackage{authblk}
\usepackage[margin=1in]{geometry}
\usepackage{tikz}
\usepackage{color}
\usepackage{datetime}
\usepackage{comment}
\DeclareMathOperator{\transverse}{\cap\kern-7.75pt\top}

\newcommand{\R}{\mathbb{R}}

\newcommand{\nit}{\mathrm}

\def\XXint#1#2#3{{\setbox0=\hbox{$#1{#2#3}{\int}$ }
\vcenter{\hbox{$#2#3$ }}\kern-.6\wd0}}

\newtheorem{theorem}{Theorem}

\newtheorem{lemma}{Lemma}

\newtheorem{rem}{Remark}

\usepackage[bookmarks=true, % true means bookmarks
            bookmarksnumbered=false, % left window are numbered
            bookmarksopen=false, % true means only level 1 are displayed. 
            colorlinks=true, 
            linkcolor=red]{hyperref}

\begin{document}
\title{A Lower Bound for the Reach of Flat Norm Minimizers}
\author[1]{Enrique G. Alvarado\thanks{ealvarado@math.wsu.edu}}
\author[1]{Kevin R. Vixie\thanks{vixie@speakeasy.net}}
\affil[1]{Department of Mathematics and Statistics, Washington State University}

\renewcommand\Authands{ and }

\maketitle

\begin{abstract}
  We establish a quantitative lower bound on the reach of flat norm
  minimizers for boundaries in $\R^2$.
\end{abstract}

\section{Introduction}

Previous work has shown that bounded curvature approximations to
shapes in $\R^2$ can be effectively computed using the fact that the
$L^1$TV functional -- for which there are many algorithms -- actually
computes the boundaries of co-dimension-$1$ boundaries. While a shape
with a curvature bound of $B$ is certainly regular in so far as
smoothness is concerned, there is another sense of regularity that is
still unsettled, at least quantitatively. And that is the question of
reach. Very briefly, the \emph{reach} of a set is the supremum of the
radii $s$ such that a ball of radius s cannot touch the boundary in
more than one place at a time. While a reach of $r$ implies the
curvature of the boundary is bounded by $\frac{1}{r}$, having the
curvature bounded by $\frac{1}{r}$ does not guarantee that the reach
is at least $r$.
\begin{figure}[htp!]
  \centering
\scalebox{0.5}{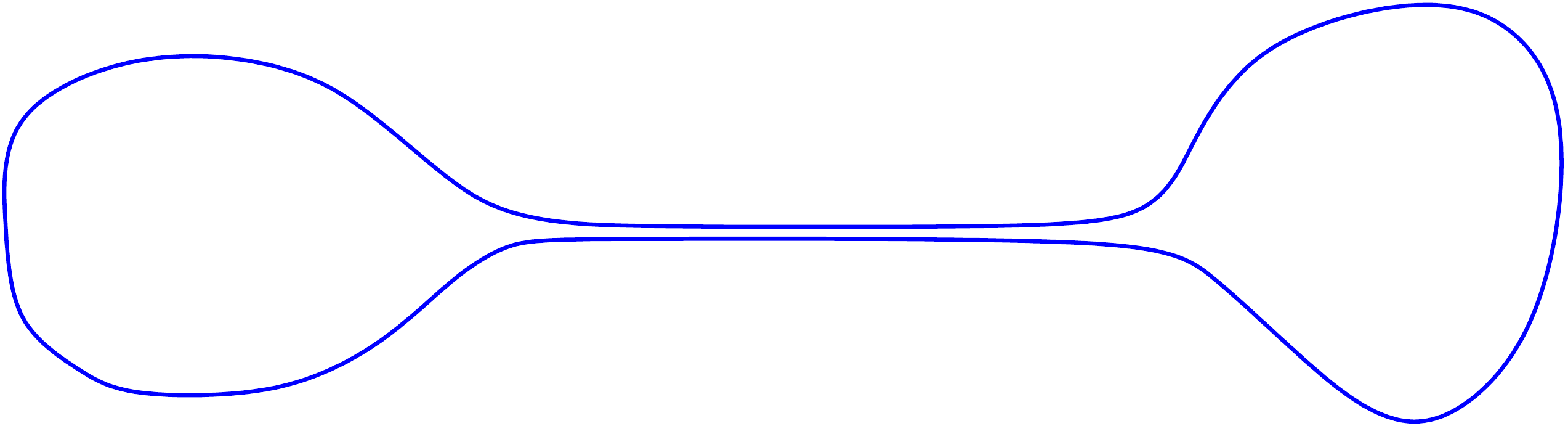}  
  \caption{Low curvature does not imply large reach.}
  \label{fig:neck}
\end{figure}
In this paper we show that while a curvature bound of B does not imply
that a flat norm minimizer has a reach of $\frac{1}{B}$, being a flat
norm minimizer does imply that the reach is at least $\frac{C}{B}$
where $C<1$ is not too small.

The previous work this paper builds on includes the paper of Morgan
and Vixie~\cite{morgan-2007-1} in which they pointed out that the
$L^1$TV functional was computing the flat norm, the Allard
paper~\cite{allard-2007-1} in which fundamental regularity results
were developed and the paper of Van Dyke and
Vixie~\cite{van-dyke-2012-thin} in which it is shown that $C < 1$.

\section{Definitions and Notation}

$\bf{2.1.\ Functions.}$\\

Let $\Gamma \subset \mathbb{R}^2$ be compact submanifold of dimension
$1$ and suppose $\Gamma = \partial\Omega$ for some open $\Omega
\subset \mathbb{R}^2$, and define the $\bf{signed\ distance\
  function}$ to $\Gamma$, $\delta_\Gamma: \mathbb{R}^2 \to \mathbb{R}$
as
$$
\delta_\Gamma(x) = \left\{
        \begin{array}{ll}
            \nit{dist}(x, \Gamma) & \quad x \in \Omega \\
            -\nit{dist}(x, \Gamma) & \quad x \in \mathbb{R}^2\setminus \Omega.
        \end{array}
    \right.
$$
If it is contextually clear, we will suppress the subscript $\Gamma$ and just write $\delta(x)$. \\

For $r > 0$, define the open $r$-neighborhood about $E$ as $U_r(E) = \{x \in \mathbb{R}^2 : \nit{dist}(x, E) < r\}$. A set $E$ is said to have $\bf{positive\ reach}$ if there exists an $r > 0$ such that each $x \in U_r(E)$ has a unique nearest point in $E$. Define reach$(E)$ to be the supremum of all such $r \in [0, \infty]$. \\

We denote any scaling of the unit circle in $\R^2$ by
$\mathcal{S}^1$. Which scaling we are using is context dependent,
determined by the constraint that embeddings we study are isometric.
Define $\gamma: \mathcal{S}^1 \to \mathbb{R}^2$ to be a $C^k$
isometric embedding of $\mathcal{S}^1$, and let $\Gamma$ denote it's
image. We define the $\bf{outer-normal\ map}$, $\alpha_\epsilon:
\mathcal{S}^1\to\mathbb{R}^2$ as $\alpha_\epsilon(s) = \gamma(s) +
\epsilon n(s)$, where $n(s)$ denotes the outer unit normal of $\Gamma$
at $\gamma(s)$. Similarly, we may define the $\bf{inner-normal\ map}$
with the inner unit normal of $\Gamma$ instead. We will denote the
image of the outer-normal map and inner-normal map as
$\Gamma_\epsilon$, and $\bar{\Gamma}_\epsilon$ respectively.\\

Let $\rho$ be the $\bf{injectivity\ radius}$ of the inner and outer normal map of $\Gamma$; that is, the supremal $\epsilon$ for which both $\alpha_\epsilon$ and $\bar{\alpha}_\epsilon$ are injective on $\mathcal{S}^1$, and notice that $\rho = \nit{reach}(\Gamma)$.\\ 

We will let cl, and int denote the ``closure'', and ``interior'', respectively.\\

Whenever $E \subset \mathbb{R}^n$ and $p$ is an accumulation point of $E$, we define the $\bf{tangent\ cone}$ of $E$ and $p$ as 
\begin{align}
\nit{Tan}(E, p) &= \bigcap_{0 < r < \infty}\nit{cl}\{t(x - p) : 0 < t < \infty\ \nit{and}\ x \in E\cap(B(p, r)\setminus \{p\})\}
\end{align}
and the $\bf{normal\ space}$ to $E$ and $p$ as 

\begin{align}
\nit{Nor}(E, p) = \bigcap_{w \in \nit{Tan}(E, p)}\{v \in \mathbb{R}^n : \langle v, w \rangle \leq 0\}.
\end{align}

$\bf{2.2.\ Currents\ and\ the\ multiscale\ flat\ norm.}$\\

For the readers convenience, we will concisely present the definitions
necessary to define the flat norm. However, for a full and complete
discussion of these topics, the reader should refer to
\cite{morgan-2008-geometric} and \cite{federer-1969-1}. Although we may define the following for any real finite dimensional vector space, we will focus on defining the following on $\mathbb{R}^n$.\\

For a non-negative integer $k \leq n$, let $\Lambda_k(\mathbb{R}^n)$ and $\Lambda^k(\mathbb{R}^n)$ denote the space of $k$-vectors and its dual space, the space of $k$-covectors in $\mathbb{R}^n$; respectively. For $\xi = v_1 \wedge v_2 \wedge ... \wedge v_k$ in $\Lambda_k(\mathbb{R}^n)$, and $\varphi = w_1 \wedge w_2\wedge ... \wedge w_k$ in $\Lambda^k(\mathbb{R}^n)$, we have the dual Euclidean norms, 

\begin{align}
|\xi| &= \sqrt{\nit{det}(V^tV)}
\end{align}
where $V$ denotes the $k\times n$ matrix with columns $v_1 , ..., v_k$, and similarly for $|\varphi|$, modulo a matrix $P$ with columns $w_1, w_2, ..., w_k$. Intuitively, the dual Euclidean norm of a $k$-(co)vector is the area of the $k$-dimensional parallelepiped defined by its vectors.\\ 

Now define the $\it{mass}$ norm for $\varphi \in \Lambda^k(\mathbb{R}^n)$ to be

\begin{align}
||\varphi|| &= \sup\{\varphi(\xi) : \xi\ \nit{is\ a\ simple\ }k\nit{-vector\ and}\ |\xi| \leq 1\},
\end{align}
where $\xi$ is defined to be a simple $k$-vector if it may be written as a single wedge product of vectors, and where
\begin{align*}
\varphi(\xi) &= \nit{det}(P^tV).
\end{align*}

Assuming that $U$ is an open subset of $\mathbb{R}^n$, let $\mathcal{E}^k(U)$ denote the real vector space of $C^\infty$ differential $k$-forms on $U$ and let $\mathcal{D}^k(U)$ be those forms which have compact support, where the support of a form $\omega \in \mathcal{E}^k(U)$ is defined to be $\nit{spt}(\omega) = \nit{cl}\{x \in \mathbb{R}^n : \omega(x) \neq 0\}$. We will define the mass of $\omega$ as 
\begin{align}
|| \omega || &= \sup\{|| \omega(x)|| : x \in U\}.
\end{align}

Now, the space of $k$-dimensional currents is the dual space of $\mathcal{D}^k(U)$ under the weak topology and will be denoted as $\mathcal{D}_k(U)$. For $T\in \mathcal{D}_k(U)$, we define its $\bf{mass}$ to be 

\begin{align}
{\bf{M}}(T) &= \sup\{T(\omega) : \omega \in \mathcal{E}^k(U)\ \nit{and}\ ||\omega|| \leq 1\}. 
\end{align}
and its $\it{boundary}$ is defined when $k \geq 1$ to be $\partial T(\omega) = T(\nit{d}\omega)$ for all $\omega \in \mathcal{D}^{k-1}(U)$. Whenever $T$ is a $0$-current, we define $\partial T = 0$. The boundary operator for currents is linear and nilpotent. That is, $\partial \partial T = 0$; and is inherited by the linear and nilpotent operator of the exterior derivative on forms, $\nit{d}\omega$. \\

Now for $\lambda > 0$, define the $\bf{multiscale\ flat\ norm}$ of a current $T \in \mathcal{E}^k$ as 

\begin{align}
\mathcal{F}_\lambda(T) &= \inf\{{\bf{M}}(A) + \lambda{\bf{M}}(S) : A \in \mathcal{E}_k, S \in \mathcal{E}_{k+1}\ \nit{and}\ T = A + \partial S\}
\end{align}
where $\mathcal{E}_k$ denotes the space of compact $k$-dimensional currents in $\mathbb{R}^n$. 

\section{Results}

\begin{lemma} Let $\Gamma$ be a $C^1$ isometric embedding of $\mathcal{S}^1$ into $\mathbb{R}^2$ and suppose $n$ is differentiable at $s\in \mathcal{S}^1$ such that $|n'(s)| < K$. Then if $0 < \epsilon < 1/K$ then $T( \Gamma, \gamma(s)) = T(\Gamma_\epsilon, \alpha_\epsilon(s))$. 
\end{lemma}

\begin{proof}
Since we are trying to show equality of our tangent spaces, all we must show is that $\alpha'(s) = c\gamma'(s)$ for some $c \neq 0$. \\

Let $n$ be differentiable at $s \in \mathcal{S}^1$. Since $|n| = 1$ on $\mathcal{S}^1$, $\langle n'(s),n(s)\rangle = 0$.  Hence, since $\langle \gamma'(s), n(s)\rangle = 0$ there exists an $r\in\mathbb{R}$ for which $n'(s) = r\gamma'(s)$. Thus $\alpha'(s) = \gamma'(s)(1 + r\epsilon)$. We must therefore show that $1 + r\epsilon \neq 0$; which reduces to showing that $\epsilon \neq 1/r$. \\

Since $|n'(s)| < K$, we get that $|r| < K$. Moreover since $0 < \epsilon < 1/K$, we get that $\epsilon < 1/|r|$ and hence 
$-1/|r| < \epsilon < 1/|r|$; implying that $\epsilon \neq -1/r$. 
\end{proof}

\begin{lemma}\label{lem:tangents-match}
If $\Gamma$ is a $C^{1,1}$ embedding of $\mathcal{S}^1$ into $\mathbb{R}^2$ and $\epsilon < 1/K$, then $T(\Gamma, \gamma(s)) = T(\Gamma_\epsilon, \alpha_\epsilon(s))$ for all $s \in \mathcal{S}^1$. 
\end{lemma}

\begin{proof}
Since we will prove a local condition on $\Gamma$ and $\Gamma_\epsilon$, we may instead investigate $\alpha_\epsilon\vert_E$ for sufficiently small open sets $E$ about points in $\mathcal{S}^1$. That is, $\Gamma\vert_E$ and $\Gamma_\epsilon\vert_E$ instead of $\Gamma$ and $\Gamma_\epsilon$ respectively. We will first argue that $\Gamma_\epsilon\vert_E$ is a $C^1$ hypersurface, and then show that $T(\Gamma\vert_E, \gamma(s)) = T(\Gamma_\epsilon\vert_E, \alpha_\epsilon(s))$ for all $s \in E$. \\

Since $\Gamma\vert_E$ is $C^{1,1}$, reach$(\Gamma\vert_E) >$ 0 \cite{lucas-1957-submanifolds}; so by \cite{krantz-1981-distance} we get that the signed distance function $\delta(x)$ to $\Gamma\vert_E$ is a $C^1$ function on the open $\epsilon$-neighborhood $U_\epsilon(\Gamma\vert_E)$ of $\Gamma\vert_E$. Hence the $\epsilon$-level-set $L_\epsilon (\delta)$ is also $C^1$; and since $L_\epsilon(\delta) = \Gamma_\epsilon\vert_E$, we get that $\Gamma_\epsilon\vert_E$ is also $C^1$. \\

Now since $\Gamma$ is $C^{1,1}$, $n$ is Lipschitz and therefore differentiable a.e in $\mathcal{S}^1$. Hence by Lemma 1, $T(\Gamma, \gamma(s)) = T(\Gamma_\epsilon, \alpha(s))$ for a.e $s\in \mathcal{S}^1$. Therefore if we pick $x \in \mathcal{S}^1$ such that $n$ is not differentiable at $x$, there exists a sequence $\{x_i\}_{i=1}^\infty$ in $E$ converging to $x$ for which $T(\Gamma, \gamma(x_i)) = T(\Gamma_\epsilon, \alpha_\epsilon(x_i))$ for all $i$. This, together with the fact that $\Gamma_\epsilon\vert_E$ is $C^1$, implies that $T(\Gamma_\epsilon, \alpha_\epsilon(x)) = T(\Gamma, \gamma(x))$.
\end{proof}

\begin{lemma}
  Suppose that $\Gamma$ is $C^{1,1}$ and therefore has positive reach
  equal to $\rho$ and also that the curvature is strictly less than $1/\rho$. Since $\Gamma$ is compact, we may assume there exist distinct $v, v' \in \mathcal{S}^1$ for which $\alpha_\rho(v) = \alpha_\rho(v')$. We will show that $T(\Gamma, \gamma(v)) = T(\Gamma, \gamma(v'))$.
\end{lemma}

\begin{proof}
  First and foremost, we have to set everything up. Let $\Gamma$ have positive reach $\rho$. Since $\Gamma$ is compact, there exist distinct $v, v' \in \mathcal{S}^1$ for which $\alpha_\rho(v) = \alpha_\rho(v')$ or $\bar{\alpha}_\rho(v) = \bar{\alpha}_\rho(v')$. Without loss of generality let's assume $\alpha_\rho(v) = \alpha_\rho(v')$. Therefore for the following we may use the distance function on subsets $U\subset \Gamma$
\begin{align*}
 \delta_U(x) = \mathrm{dist}(x, U).   
\end{align*}

For any $\epsilon > 0$, we may define $\Gamma_{v,\epsilon} := \alpha_\rho((v - \epsilon, v + \epsilon))$ and $\Gamma_{v',\epsilon} := \alpha_\rho((v'-\epsilon, v'+\epsilon))$; for notational convenience, let $\delta_v = \delta_{\Gamma_{v,\epsilon}}$ and $\delta_{v'} = \delta_{\Gamma_{v',\epsilon}}$. Lets also denote $(v - \epsilon, v + \epsilon)$ and $(v' - \epsilon, v' + \epsilon)$ as $V_\epsilon$ and $V'_\epsilon$ respectively.\\

Notice that by the equality of injectivity radius and reach, for any $\sigma \leq \rho$ and any $x \in \alpha[(0, \sigma)\times V_\epsilon]$ there exists an $s \in \Gamma_v$ such that $|x - s| = \delta(x)$, and hence $\delta_{v}(x) = \delta(x)$. Similarly, for any $x \in \alpha[(0, \sigma)\times V'_\epsilon]$ there exists an $s' \in \Gamma_{v'}$ for which $|x - s'| = \delta(x)$ and hence $\delta_{v'}(x) = \delta(x)$.\\

By way of contradiction, suppose the tangent spaces are not equal. Therefore, $E_\epsilon := \alpha[[0,\rho]\times V_\epsilon]\cap \alpha[[0,\rho]\times V'_\epsilon]$  will have a non-empty interior, and hence there exists $\nu_2 > 0$ such that given any $0 < \epsilon < \nu_1$, for any point $x \in \alpha_\rho(\alpha|_{V}^{-1}(E_\epsilon))$, $\delta_{v}(x) > \delta_{v'}(x)$, and for any point $x \in \alpha_\rho(\alpha|_{V'}^{-1}(E_\epsilon))$, $\delta_{v}(x) < \delta_{v'}(x)$. Therefore, if we define a new function $\tilde{\delta}: \mathbb{R}^2 \to \mathbb{R}$ as 
 \begin{align*}
  \tilde{\delta}(x) &:= \delta_{v}(x) - \delta_{v'}(x),   
 \end{align*}

 \noindent we find that $\tilde{\delta}(x)$ is positive whenever $x \in \alpha_\rho(\alpha|_{V}^{-1}(E_\epsilon))$ and negative whenever $x \in \alpha_\rho(\alpha|_{V'}^{-1}(E_\epsilon))$.\\

 Now, let $\hat{x} \in \alpha_\rho(\alpha|_{V}^{-1}(E_\epsilon))$ and $\hat{x}' \in \alpha_\rho(\alpha|_{V'}^{-1}(E_\epsilon))$. Since $\tilde{\delta}$ is continuous, there exists $r \in (0,1)$ such that for $z := r\hat{x} + (1-r)\hat{x}'$, $\tilde{\delta}(z) = 0$, and hence $\delta_v(z) = \delta_{v'}(z)$.\\

 We let $s$ and $s'$ in $\Gamma_v$ and $\Gamma_{v'}$ respectively such that $\delta_{v'}(z) = |s' - z| = \delta(z) = |s - z| = \delta_{v}(z) < \rho$. Since there exists a $\nu_2 > 0$ such that for any $0 < \epsilon < \nu_2$, $\Gamma_{v,\epsilon} \cap \Gamma_{v', \epsilon} = \emptyset$, by letting $\epsilon < \min{\{\nu_1, \nu_2\}}$, we get a contradiction with $z$ having a unique closest point in $\Gamma$. 
\end{proof}

\begin{rem}
  At first glance, the proof might appear to not use the fact that
  curvature of $\Gamma$ is uniformly bounded above by something
  strictly less than $1/\rho$. However, in fact we do; we need this
  condition in order for $T(\alpha(V_\epsilon, \rho), \alpha(v,
  \rho))$ and $T(\alpha(V'_\epsilon, \rho), \alpha(v', \rho))$ to be
  tangent lines for small enough $\epsilon$. This is implied by
  Lemma~\ref{lem:tangents-match}
\end{rem}

\begin{figure}[htp!]
  \centering
\scalebox{1}{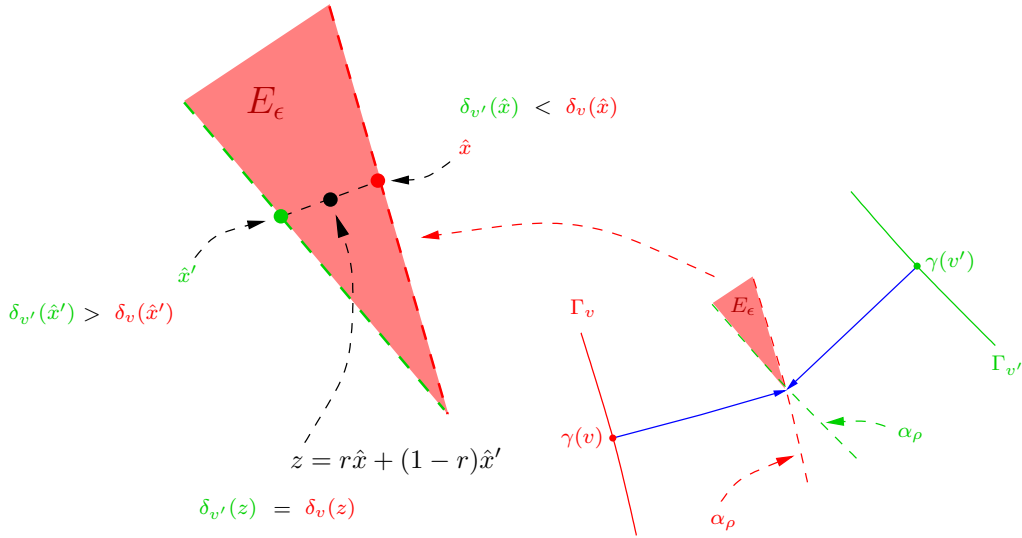}  
  \caption{By expanding $\Gamma$ with the correct normal map by $\rho$ amount, we assume that the tangent lines are not equal, and then show that there is a point $z$ inside $E_\epsilon$, and two corresponding points, $s$ and $s'$ in $\Gamma_{v}$ and $\Gamma_{v'}$ for which $|z - s'| = |z - s| = \delta(z)$ and is less than $\rho$ contradicting the fact that $\rho$ is the reach.}
  \label{fig:neck}
\end{figure}

% \begin{lemma}
%   Let $\Gamma\subset\R^2$ be a $C^{1,1}$ curve with curvature bounded by
%   $\lambda$. Suppose that there is a point $x\in\R^2\setminus\Gamma$
%   and two distinct points on $\Gamma$, $\gamma(t)$ and $\gamma(t')$,
%   such that:
%   \begin{enumerate}
%   \item $\{t,t'\}\subset \argmin_s |x-\gamma(s)|$,                 
%   \item $|x-\gamma(t)| = |x-\gamma(t')| = \rho$,          
%   \item $\rho < \frac{1}{\lambda}$ and                    
%   \item $T(\Gamma,\gamma(t)) \neq T(\Gamma, \gamma(t'))$. 
%   \end{enumerate}
% Then the reach of $\Gamma < \rho$.
% \end{lemma}

% \begin{proof}

% \end{proof}

\begin{lemma}
Given $\Gamma = T - \partial S$, if there exists a current $S^\ast$ for which $\mathbf{M}(\Gamma) > \mathbf{M}(\Gamma - \partial S^\ast) + \lambda\mathbf{M}(S^\ast)$ then $\mathbf{M}(\Gamma) + \lambda\mathbf{M}(S) > \mathbf{M}(T - \partial(S + S^\ast)) + \lambda\mathbf{M}(S + S^\ast)$.
\label{thm:lem-1}
\end{lemma}

\begin{proof}
\begin{align*}
\mathbf{M}(\Gamma) + \lambda\mathbf{M}(S) &= \mathbf{M}(T - \partial S) + \lambda\mathbf{M}(S)\\
&> \mathbf{M}(\Gamma - \partial S^\ast) + \lambda \mathbf{M}(S^\ast) + \lambda\mathbf{M}(S)\\
&\geq \mathbf{M}(\Gamma - \partial S^\ast) + \lambda\mathbf{M}(S^\ast + S)\\
&\geq \mathbf{M}(T-\partial S - \partial S^\ast) + \lambda\mathbf{M}(S^\ast + S)\\
&\geq \mathbf{M}(T - \partial (S + S^\ast)) + \lambda\mathbf{M}(S^\ast + S)
\end{align*}
\end{proof}

\begin{theorem}
\label{thm:one-circle} 
If $\Gamma_\lambda$ is a $C^{1,1}$ isometric embedding of $\mathcal{S}^1$ into $\mathbb{R}^2$, then the reach of $\Gamma_\lambda$ is bounded bellow by $C/\lambda$, where $C \approx .22$.  
\end{theorem}

\begin{proof} Let $\Gamma$ be a $C^{1,1}$ isometric embedding of $\mathcal{S}^1$ into $\mathbb{R}^2$. Suppose by way of contradiction that reach$(\Gamma) < C/\lambda$. By a comparison argument, we will show that there is a local perturbation of our current $\Gamma$ into $\Gamma^\ast := \Gamma - \partial S^\ast$ for which 
\begin{align}
\mathbf{M}(\Gamma) &> \mathbf{M}(\Gamma^\ast) + \lambda\mathbf{M}(S^\ast),
\end{align}
to then conclude, by $Lemma$~\ref{thm:lem-1} that $\Gamma$ was not a minimizer in the first place.\\

(Step 1) First, define $\alpha_\epsilon$ and $\bar{\alpha_\epsilon}$ be the outer and inner-normal map of $\Gamma$ as is defined in the first section. Similarly, we let $\rho$ be the injectivity radius of the inner and outer normal map of $\Gamma$; that is, the supremal $\epsilon$ for which both $\alpha_\epsilon$ and $\bar{\alpha}_\epsilon$ are injective on $\mathcal{S}^1$.\\  

Since reach$(\Gamma) > 0$ \cite{lucas-1957-submanifolds} and $\Gamma$ is compact, we may assume that $\rho > 0$ and that there exists distinct $v, v' \in \mathcal{S}^1$ for which $\alpha_\rho(v) = \alpha_\rho(v')$ or $\bar{\alpha}_\rho(v) = \bar{\alpha}_\rho(v')$. Without loss of generality, for the remainder of the proof we will assume $\alpha_\rho(v) = \alpha_\rho(v')$.\\

% {\red
% Since we are assuming by way of contradiction that reach$(\Gamma) < C/\lambda$, pick $0 < \rho < C/\lambda$. Now, since the curvature of $\Gamma$ is bounded by $\lambda$ $\cite{allard-2007-1}$, the embedded tangent spaces $T(\alpha_\rho(v), \Gamma_\rho)$ and $T(\alpha_\rho(v'), \Gamma_\rho)$ are equal, and hence by $\it{Lemma\ 2.}$, the embedded tangent spaces $T(\gamma(v), \Gamma)$ and $T(\gamma(v'), \Gamma)$ are parallel.\\

% This is not correct -- that is, you have to prove that the tangent spaces are parallel. But it is true and there is a proof that is easy actually. If $T(\Gamma, \gamma(v)) \neq  T(\Gamma, \gamma(v'))$ and $v$ and $v'$ are points that are $\rho$ away from some one point, then we can show that reach $< \rho$ ...\\
% }

(Step 2) Let's construct $S^\ast(t)$ to then later find an optimal $t'$ so we can then define $S^\ast = S^\ast(t')$. \\

Without loss of generality, we will work under the translation $\Gamma - \gamma(v)$ and the rotation for which $T(\Gamma, \gamma(v)) - \gamma(v)$ is equal to a coordinate axis and $T(\Gamma, \gamma(v')) - \gamma(v')$ is equal to $\mathbb{R}\times\{-2\rho\}$. \\

Now for any $t\in [-\lambda^{-1}, \lambda^{-1}]$, lets define the following two regions, (Figure 1).\\

$\bullet$ $R_1(t) = Q_1\setminus [B_{1/\lambda}((0, \lambda^{-1}))\cup B_{\lambda^{-1}}((0, -\lambda^{-1}))]$ where $Q_1$ is the closed rectangle defined by the four vertices $(-t, \lambda^{-1}), (-t, -\lambda^{-1}), (t, \lambda^{-1})$, and $(-t, -\lambda^{-1})$, and\\

$\bullet$ $R_2(t) = Q_2\setminus [B_{\lambda^{-1}}((0, \lambda^{-1} - \rho))\cup B_{\lambda^{-1}}((0, -\lambda^{-1} - \rho))]$ where $Q_2$ is the closed rectangle defined by the four vertices $(-t, \lambda^{-1} - \rho), (-t, -\lambda^{-1} - \rho), (t, \lambda^{-1} - \rho)$, and $(-t, -\lambda^{-1} - \rho)$. \\

The curvature of $\Gamma$ is defined almost everywhere in $\mathcal{S}^1$ and is bounded above by $\lambda$ $\cite{allard-2007-1}$. Thus, for any $s \in \mathcal{S}^1$ and any $x \in (s - \pi\lambda^{-1}, s + \pi\lambda^{-1})$, $\gamma(x)$ is contained outside the two unique balls of radius $\lambda^{-1}$  which have first order contact with $T(q, \Gamma)$.\\

Therefore for $\delta \in (0, \frac{\pi}{2\lambda})$, there exists functions $F$ and $F'$ defined on the first coordinate axis for which the images of $[v - \delta, v + \delta]$ and $[v' - \delta, v' + \delta]$ under $\gamma$ equal the graph of $F$ and $F'$ respectively. Hence, for any point $t \in [-\lambda^{-1}, \lambda^{-1}]$ $(t,F(t))$ and $(t,F'(t))$ are contained inside the regions $R_1(t)$ and $R_2(t)$ respectively.(See again Figure 1)\\

\begin{figure}[htp!]
  \centering
\scalebox{1}{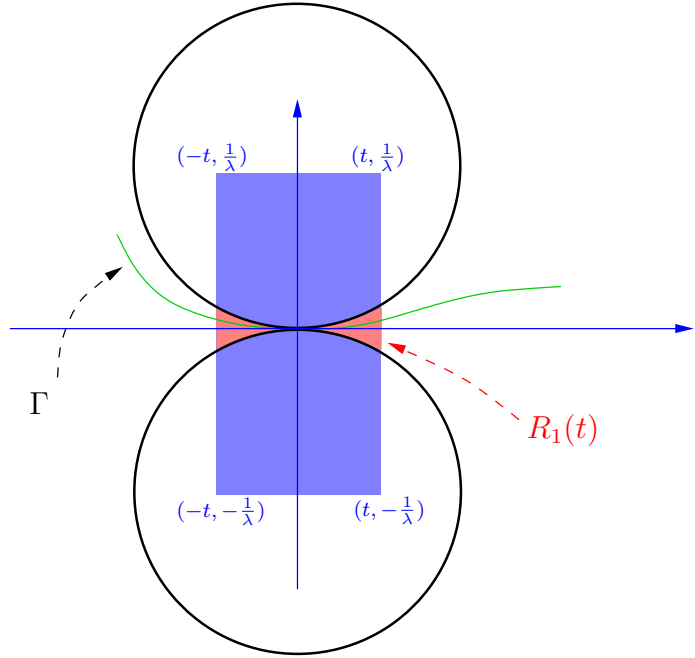}  
  \caption{The region $R_1(t)$ is constructed by taking the interior of the two disks away from the rectangle.}
  \label{fig:neck}
\end{figure}

For any $t \in [-\lambda^{-1}, \lambda^{-1}]$, let
\begin{align*}
  \Gamma^1(t) &:= \{(x, F(x)) : x\in [-t\lambda^{-1}, t\lambda^{-1}]\}
\end{align*}

and
\begin{align*}
  \Gamma^2(t) &:= \{(x, F'(x)) : x\in [-t\lambda^{-1}, t\lambda^{-1}]\}
\end{align*}

be the local images of $\Gamma$ contained in $R_1$ and $R_2$ respectively. For $0 < s \leq t$, let $L(s)$ denote the portion of the vertical line segment $x = s$ that lies between $\Gamma^1(t)$ and $\Gamma^2(t)$, and let $S^\ast(t)$ be the bounded region whose boundary is the union $\Gamma^1(t)\cup \Gamma^2(t)\cup L(-t)\cup L(t)$ with orientation induced by $\Gamma$.\\

(Step 3) Finding the lower bound for reach$(\Gamma)$.\\

\begin{figure}[htp!]
  \centering
\scalebox{.9}{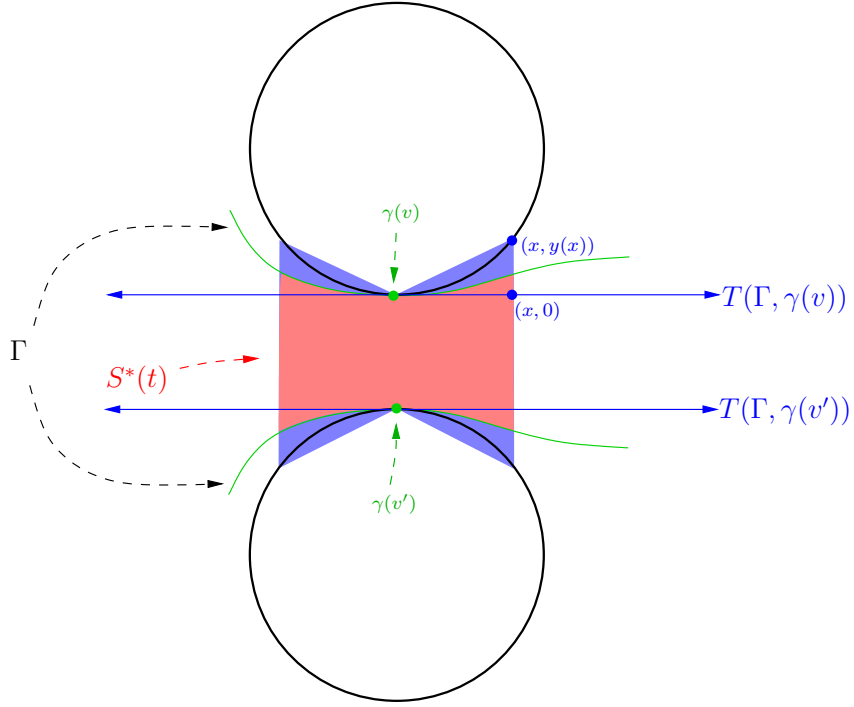}  
  \caption{The area of the region $S^{\ast}(t)$ will be bounded above by the area of the butterfly wings region.}
  \label{fig:neck}
\end{figure}

For all $x \in (0, \lambda^{-1})$, we have the lower bound

\begin{align}\label{eq:step3-eq1}
 \mathbf{M}(\Gamma^1(x)+ \Gamma^2(x)) &\geq 4x
 \end{align} 
and the upper bound

\begin{align}\label{eq:step3-eq2}
(4y +2(2\rho)) + \lambda(2(2\rho x + xy)) &\geq \mathbf{M}(L(x) + L(-x)) + \lambda\mathbf{M}(S^\ast(x)). 
 \end{align}
 
Therefore, since $\mathbf{M}(\Gamma) = \mathbf{M}(\Gamma^1(x) + \Gamma^2(x)) + \mathbf{M}(\Gamma - (\Gamma^1(x) + \Gamma^2(x)))$, and $\mathbf{M}(\Gamma - \partial S^\ast) = \mathbf{M}(\Gamma - (\Gamma^1(x) + \Gamma^2(x))) + \mathbf{M}(L(x) + L(-x))$, showing that $\mathbf{M}(\Gamma) > \mathbf{M}(\Gamma-\partial S^\ast(x)) + \lambda\mathbf{M}(S^\ast(x))$ reduces to showing 

\begin{align}\label{eq:step3-eq3}
\mathbf{M}(\Gamma^1(x) + \Gamma^2(x)) &> \mathbf{M}(L(x) + L(-x)) + \lambda\mathbf{M}(S^\ast(x)).
\end{align}

Since\ \ref{eq:step3-eq1} and\ \ref{eq:step3-eq2} are true for all $x \in (0, \lambda^{-1})$, we will show\ \ref{eq:step3-eq3} by finding values of $x$ such that for $y(x) = \sqrt{(1/\lambda)^2 - x^2} - 1/\lambda$, we have

\begin{align}
4x &> (4y +2(2\rho)) + \lambda(2(2\rho x + xy));
\end{align}
 or equivalently,

\begin{align}\label{eq:step3-eq4}
\displaystyle\frac{4x - 2y(\lambda x + 2)}{4(\lambda x + 1)} &> \rho.
\end{align}

By changing to polar coordinates,\ \ref{eq:step3-eq4} is equivalent to

\begin{align}\label{eq:step3-eq5}
\displaystyle\left(\frac{1}{\lambda}\right)\frac{2\cos{\theta} - (1 + \sin{\theta})(\cos{\theta} + 2)}{2(\cos{\theta} + 1)} &> \rho.
\end{align}

Indeed, we may choose any $\theta \in (3\pi/2, 2\pi)$ for which\ \ref{eq:step3-eq5} holds true, however, since we are trying to obtain the largest lower bound possible, we find that for 

\begin{align}
C(\theta) &= \frac{2\cos{\theta} - (1 + \sin{\theta})(\cos{\theta} + 2)}{2(\cos{\theta} + 1)},
\end{align}
we get that
\begin{align*}
\hat{C}&:= \sup_{\theta \in (3\pi/2, 2\pi)}C(\theta)\\
&\approx .2217
\end{align*}
and that 

\begin{align*}
\theta' &:= \nit{argmin}\{C(\theta) : \theta \in (3\pi/2, 2\pi)\}\\
&\approx 5.231. \\
\end{align*}

Therefore, letting $S^\ast = S^\ast(x')$ for $x' = \lambda^{-1}\cos(\theta')$ concludes the proof.
\end{proof}

\begin{theorem}
  Suppose that $T = \partial\Omega$ for $\Omega\subset\R^2$ and that
  $\Gamma = T-\partial S$ is a corresponding one dimensional
  multiscale-flatnorm minimizer with scale parameter $\lambda$. Then
  the reach of $\Gamma$ is bounded below by $\hat{C}/\lambda$.
\end{theorem}

\begin{proof}
  Since $\Gamma$ is a $C^{1,1}$ embedded curve in $\R^2$ with empty
  boundary, it has positive reach. This implies that there is some
  small $\epsilon$ so that every component of $\Gamma$ contains a
  translate of $B(0,\epsilon)$. Since the length of $\gamma$ is finite,
  we must also have that $\Gamma\subset B(0,N)$ for some large enough
  $N\in\Bbb{N}$. Since each of the balls in separate components of
  $\Gamma$ are disjoint, then there are most ${\frac{\pi N^2}{\pi
      \epsilon^2}}$ components of $\Gamma$. Since each of the
  components of $\Gamma$ are embedded circles, each with the same
  orientation, we get that the comparison construction in
  Theorem~\ref{thm:one-circle} goes through whether or not the points
  $v$ and $v'$ are on the same circle or not.  Therefore, we can use
  the argument used to prove Theorem~\ref{thm:one-circle} to finish
  the proof.
\end{proof}

\bibliographystyle{plain}
\bibliography{ea_krv_1}
\end{document}